\documentclass[11pt]{amsart}

\usepackage{amssymb,amsmath,amsfonts,amstext,latexsym}
\usepackage[all]{xy}
\usepackage[colorlinks=true, urlcolor=rltblue, citecolor=drkgreen, linkcolor=drkred] {hyperref}
\usepackage{color}
\usepackage{pdfsync}
\usepackage{enumitem}
\definecolor{rltblue}{rgb}{0,0,0.4}
\definecolor{drkgreen}{rgb}{0,0.4,0}
\definecolor{drkred}{rgb}{0.5,0,0}

\newtheorem{thm}{Theorem}
\newtheorem{lemma}[thm]{Lemma}
\newtheorem{proposition}[thm]{Proposition}

\newtheorem{theorem}[thm]{Theorem}
\newtheorem{corollary}[thm]{Corollary}

\theoremstyle{definition}
\newtheorem{definition}[thm]{Definition}

\theoremstyle{remark}

\newtheorem{historic}[thm]{Historic Remark}

\theoremstyle{plain}

\newcounter{contenumi}

\def\D{{\mathcal D}}

\def\upto{\mathop{\upharpoonright}}

\def\and{\mathrel{\&}}

\def\sminus{\smallsetminus}

\def\isom{\cong}
\def\Si{\Sigma}

\def\A{\mathcal{A}}

\def\B{\mathcal{B}}
\def\C{\mathcal{C}}
\def\I{\mathcal{I}}

\def\F{\mathcal{F}}

\def\om{\omega}
\def\bbar{\bar{b}}

\def\si{\sigma}

\newcommand \converges{\downarrow}

\def\A{\mathcal A}

\def\xbar{{\bar{x}}}

\def\cbar{{\bar{c}}}

\def\abar{{\bar{a}}}
\def\bbar{{\bar{b}}}
\def\xbar{{\bar{x}}}
\def\dbar{{\bar{d}}}
\def\ctt{{\mathtt c}}

\def\Sic{\Si^\ctt}
\def\Dec{\Delta^\ctt}
\def\Pic{\Pi^\ctt}

\def\CC{{\mathfrak C}}

\def\Ahat{\hat{\A}}

\def\Si{\Sigma}

\def\om{\omega}

\def\Bbar{\bar{B}}

\def\HF{\mathbb{HF}}

 \def\Ahat{\hat{\A}}
 \def\Bhat{\hat{\B}}

\def\Do{\mathcal Dom}

\newcounter{been}

\newcounter{eebeen}

\def\tauti{\tilde{\tau}}

\def\DgSp{DgSp}

\newcommand{\Iso}[1]{\text{Iso}(#1)}
\newcommand{\DD}{\mathfrak{D}}
\renewcommand{\Ahat}{\widehat{\mathcal{A}}}
\renewcommand{\Bhat}{\widehat{\mathcal{B}}}

\newcommand{\Atilde}{\widetilde{\mathcal{A}}}
\newcommand{\Btilde}{\widetilde{\mathcal{B}}}

\newcommand{\ghat}{\widehat{g}}

\newcommand{\comment}[1]{}
\newcommand{\res}{\upto}
\DeclareMathOperator{\id}{id}
\newcommand{\Dom}[2]{\Do_{#2}^{#1}}

\def\Fra{\mathfrak{F}}

\title{Computable Functors and Effective Interpretability}

\author[M. Harrison-Trainor]{Matthew Harrison-Trainor}
\address{Group in Logic and the Methodology of Science\\
University of California, Berkeley\\
 USA}
\email{matthew.h-t@math.berkeley.edu}
\urladdr{\href{http://math.berkeley.edu/~mattht}{http://math.berkeley.edu/$^\sim$mattht}}

\author[A. Melnikov]{Alexander Melnikov}
\address{The Institute of Natural and Mathematical Sciences\\
Massey University\\
 New Zealand}
\email{alexander.g.melnikov@gmail.com}
\urladdr{\href{https://dl.dropboxusercontent.com/u/4752353/Homepage/index.html}{https://dl.dropboxusercontent.com/u/4752353/Homepage/index.html}}

\author[R. Miller]{Russell Miller}
\address{Mathematics Dept., Queens College; Ph.D.\ Programs in Mathematics \& Computer Science, Graduate Center\\
City University of New York\\
USA}
\email{Russell.Miller@qc.cuny.edu}
\urladdr{\href{http://qcpages.qc.cuny.edu/~rmiller}{http://qcpages.qc.cuny.edu/$^\sim$rmiller}}

\author[A. Montalb\'an]{Antonio Montalb\'an}
\address{Department of Mathematics\\
University of California, Berkeley\\
	USA}
\email{antonio@math.berkeley.edu}
\urladdr{\href{http://www.math.berkeley.edu/~antonio/index.html}{www.math.berkeley.edu/$^\sim$antonio}}

\thanks{The first author was partially supported by the Berkeley Fellowship and NSERC grant PGSD3-454386-2014.
The second author was supported by the Packard Foundation.
The third author was supported by NSF grants \# DMS-1362206 and DMS-1001306, and by several PSC-CUNY research awards.
The fourth author was partially supported by the Packard Fellowship and NSF grant \# DMS-1363310.
This work took place in part at a workshop held by the Institute for Mathematical Sciences of the National University of Singapore.}

\begin{document}

\begin{abstract}
Our main result is the equivalence of two notions of reducibility between structures.
One is a syntactical notion which is an effective version of interpretability as in model theory, and the other one is a computational notion which is a strengthening of the well-known Medvedev reducibility.
We extend our result to effective bi-interpretability and also to effective reductions between classes of structures.\end{abstract}

\maketitle


%

\section{Introduction}

The main purpose of this paper is to establish a connection between two standard methods of computable structure theory for reducing one structure into another one.
One of this methods, effectively interpretability (Definition \ref{def: eff int}), is purely syntactical and is an effective version of the classical notion of interpretability in model theory .
It is equivalent to the well-studied notion of $\Sigma$-reducibility.
The other method is purely computational, and it involves computing copies of one structure from copies of the other using what we will call computable functors (Definition \ref{def:functors}).

In computable structure theory we study complexity issues related to mathematical structures.
One of the objectives of the subject is to measure the complexity of structures.
There are three commonly used methods to compare the complexity of structures: Muchnik reducibility, Medvedev reducibility and $\Sigma$-reducibility. 
The first two are computational, in the sense that they are about copies of a structure computing other copies; while the third one is purely syntactical.
They are listed from weakest to strongest and none of the implications reverse (as proved by Kalimullin \cite{Kal09}).

\subsection*{Effective interpretability} 

Informally, a structure $\A$ is {\em effectively interpretable} in a structure $\B$ if there is an interpretation of $\A$ in $\B$ (as in model theory \cite[Definition 1.3.9]{MarkerBook}), but where the domain of the interpretation is allowed to be a subset of $\B^{<\om}$ (while in the classical definition it is required to be a subset of $\B^n$ for some $n$), and where all sets in the interpretation are required to be ``computable within the structure'' (while in the classical definition they need to be first-order definable).\footnote{We remark that this definition is slightly different from what the fourth author called effective-interpretability in \cite[Definition 1.7]{MonFixed}, as we now allow the domain to be a subset of $\B^{<\om}$ rather than $\B^n$ for some $n$, and we do not allow parameters in the definitions.}
Here, by  ``computable within the structure''  we mean uniformly relatively intrinsically computable (see Definition \ref{def:rice}).
Effective interpretability is among the strongest notions of reducibility between structures that are usually considered.
It gives a very concrete way of producing the structure $\A$ from the structure $\B$,
and hence implies that essentially any kind of information encoded in $\A$ is also encoded in $\B$.

Effective interpretability is equivalent to the parameterless version of the notion of $\Si$-definability, introduced by Ershov \cite{Ershov96} and widely studied in Russia over the last twenty years (for instance \cite{Puz09, StuEMU, MK08, Kal09}).
The standard definition of $\Si$-definability is quite different in format: it uses  the first-order logic over $\HF(\B)$, the structure of hereditarily finite sets over $\B$,  instead of the computably infinitary language over  $\B^{<\om}$.
For a more detailed discussion of the equivalence between effective interpretability and $\Sigma$-definability see \cite[Section 4]{MonRice}.

Before giving the formal definition, we need to review one more concept. 

\begin{definition}\label{def:rice}
A relation $R$ on $\A^{<\om}$ is said to be {\em uniformly relatively intrinsically computably enumerable (u.r.i.c.e.)} if there is a c.e.\ operator $W$ such that for every copy $(\B,R^\B)$ of $(\A,R)$, $R^\B= W^{D(\B)}$.
A relation $R$ on $\A^{<\om}$ is said to be {\em uniformly relatively intrinsically computable (u.r.i.\ computable)} if there is a computable operator $\Psi$ such that for every copy $(\B,R^\B)$ of $(\A,R)$, $R^\B= \Psi^{D(\B)}$.

(Here  $D(\B)$ referee to the atomic diagram of $\B$; it is an infinite binary sequence that encodes the truth of all the atomic facts about $\Bhat$.  See \cite[Section 2]{MonRice}, for instance, for a formal definition).
\end{definition}

These relations are the analogues of the c.e.\ and computable subsets of $\om$ when we look at relations on a structure.
They are computability theoretic notions, but they can be characterized in purely syntactical terms:
It follows from the results in Ash, Knight, Manasse and Slaman \cite{AKMS89}, and Chisholm \cite{Chi90} that a relation $R$ is  u.r.i.c.e.\ if and only if it can be defined by a computably infinitary $\Si_1$ formula without parameters;
a relation $R$ is u.r.i.\ computable if both it and its complement can be defined by computably infinitary $\Si_1$ formulas without parameters.
(We will use $\Sic_1$ to denote the computably infinitary $\Si_1$ formulas, and the same for $\Dec_1$, $\Pic_1$, etc.)
These theorems were originally proved for $R\subseteq \A^n$ for some $n$, but they also hold for $R\subseteq \A^{<\om}$ (see \cite[Theorem 3.14]{MonRice}).
In this latter case, we say that $R$ is $\Sic_1$-definable if there is a computable list $\varphi_1,\varphi_2, \varphi_3,...$ of $\Sic_1$ formulas defining  $R\cap \A^1$, $R\cap \A^2$, $R\cap \A^3$,... respectively.
The use of $\A^{<\om}$ is not just to be able to take  subsets of the different $A^n$ at the same time.
Traditionally, computability theory is usually developed by considering subsets of $\om$ and this is workable because every finite object can be coded by a natural number.
In the same way,
when we are talking about computability over a structure, $\A^{<\om}$ is the simplest domain where we can develop computability without losing generality.
For instance, it is not hard to see that we can easily encode subsets of $(\A^{<\om})\times\om$ by subsets of $\A^{<\om}$ in an effective way\footnote{For example, $(b_0, \dots, b_k, m)$ can be coded by the definable class of tuples of the form $(b', b_0, \ldots, b_k, b', \ldots ($m  times$) \ldots, b')$ where $b' \neq b_i$ for all $i$. Different choices of $b'$ will code the same tuple, but we can identify all such codes later when we introduce a definable equivalence relation upon the domain.} so that we can talk about r.i.c.e.\ subsets of $(\A^{<\om})\times\om$, etc.
Thus, we say that a sequence of relations $(R_i:i\in\om)$ where $R_i\subseteq\A^{<\om}$ is r.i.c.e.\ or $\Sic_1$-definable if it is as a subset of $(\A^{<\om})\times\om$.

Throughout the rest of the paper, we assume that all our structures have a computable language. Without loss of generality, we may further assume that all  languages considered are relational.

\begin{definition}\label{def: eff int}
We say that a structure $\A = (A; P_0^\A,P_1^\A,...)$ (where $P_i^\A\subseteq A^{a(i)}$) is {\em effectively interpretable} in $\B$ if there exist a $\Dec_1$-definable (in the language of $\B$, without parameters) sequence of relations $(\Dom{\A}{\B}, \sim, R_0, R_1,...)$ such that
\begin{enumerate}
	\item $\Dom{\B}{\A}\subseteq \B^{<\om}$,
	\item $\sim$  is an equivalence relation on $\Dom{\B}{\A}$, 
	\item $R_i\subseteq (B^{<\om})^{a(i)}$ is closed under $\sim$ within $\Do_\A^\B$,
\end{enumerate}
and there exists a function $f^\B_\A\colon \Dom{\B}{\A} \to \A$ which induces an isomorphism: 
\[
(\Dom{\B}{\A}/\sim; R_0/\sim,R_1/\sim,...) \isom (A; P_0^\A,P_1^\A,...),
\]
where $R_i/ \sim$ stands for the $\sim$-collapse of $R_i$.\footnote{
In previous definitions in the literature, $\Dom{\B}{\A}$ was asked to be $\Sic_1$ definable instead of $\Dec_1$ definable (see for instance \cite[Definition 1.7]{MonFixed} and \cite[Definition 5.1]{MonICM}).
But in fact one can demonstrate these definitions  are equivalent.
Indeed,  given a $\Sigma$-interpretation, with the domain consisting of the tuples $\bar{x}$ satisfying a countable disjunction of formulas $\exists \bar{s} \varphi_i(\bar{x},\bar{s})$, we find a new domain consisting of the tuples $(\bar{x},\bar{s},i)$  with $(\bar{x},\bar{s})$ satisfying $\varphi_i$, and we have $(\bar{x},\bar{s},i)$ equivalent to $(\bar{y},\bar{t},j)$ iff $\bar{x}$ and $\bar{y}$ are equivalent in the $\Sigma$-interpretation.
}
\end{definition}

As important as the notions of reducibility between structures are the notions of equivalence between structures.
Despite extensive study of effective interpretability, or $\Sigma$-definability, over the last couple of decades, the associated notion of bi-interpretability has not been considered until recently \cite[Definition 5.2]{MonICM}.
Let us remark that the notion of $\Sigma$-equivalence between structures, which says that two structures are $\Sigma$-definable in each other, has been studied (\cite{StuEMU}), but the notion of bi-interpretability we are talking about is much stronger.
Informally:
two structures $\A$ and $\B$ are {\em effectively bi-interpretable} if they are effectively interpretable in each other, and furthermore, the compositions of the interpretations are $\Dec_1$-definable in the respective structures.
In other words, when two structures interpret each other, we have that $\A$ can be interpreted as a structure inside $\B^{<\om}$, and that $\B^{<\om}$ can be interpreted as a structure inside $(\A^{<\om})^{<\om}$.
Thus, we have an interpretation of $\A$ inside $(\A^{<\om})^{<\om}$.
For bi-interpretability, we require that the isomorphism between $\A$ and its interpretation inside $(\A^{<\om})^{<\om}$ be $\Dec_1$-definable, and the same for the isomorphism between $\B$ and its interpretation inside $(\B^{<\om})^{<\om}$.

\begin{definition}
\label{defn:eff-biinterpretable}
Two structures $\A$ and $\B$ are {\em effectively bi-interpretable} if there are effective
interpretations of each structure in the other as in Definition \ref{def: eff int} such that the compositions
\[
f^\A_\B \circ \tilde{f}^\B_\A\colon \Do_\B^{(\Do_\A^\B)} \to \B 
\quad\mbox{ and }\quad
f^\B_\A \circ \tilde{f}^\A_\B\colon \Do_\A^{(\Do_\B^\A)} \to \A 
\]
are  u.r.i.\ computable in $\B$ and $\A$ respectively.
(Here $\Do_\B^{(\Do_\A^\B)}\subseteq (\Do_\A^\B)^{<\om}$, and $\tilde{f}^\B_\A\colon (\Do_\A^\B)^{<\om}\to \A^{<\om}$ is the obvious extension of $f^\B_\A\colon \Do_\A^\B\to \A$ mapping $\Do_\B^{(\Do_\A^\B)}$ to $\Do_\B^\A$.)
\end{definition}

When two structures are effectively bi-interpretable, they look and feel the same from a computability point of view.
In \cite[Lemma 5.3]{MonICM} the fourth author shows that if $\A$ and $\B$ are effectively bi-interpretable then: 
they have the same degree spectrum;
they have the same computable dimension;
they have the same Scott rank;
their index sets are Turing equivalent (assuming the structures are infinite);
$\A$ is computably categorical if and only if $\B$ is;
$\A$ is rigid if and only if $\B$ is;
$\A$ has the  c.e.\ extendability condition if and only if $\B$ does;
for every $R\subseteq\A^{<\om}$, there is a $Q\subseteq\B^{<\om}$ which has the same relational degree spectrum, and vice-versa;
and the jumps of $\A$ and $\B$ are effectively bi-interpretable too.

\subsection*{Computable functors}

One of the most common ways of describing the computational complexity of a structure is by its degree spectrum.
Associated with the degree spectrum is the notion of Muchnik reducibility:
A structure $\A$ is {\em Muchnik reducible} to a structure $\B$ if every copy of $\B$ computes a copy of $\A$, or (equivalently for non-trivial structures) if $\DgSp(\A)\subseteq\DgSp(\B)$.
The uniform version of this reducibility is called Medvedev reducibility:
A structure $\A$ is {\em Medvedev reducible} to a structure $\B$ if there is a Turing functional $\Phi$ that, given a copy of $\B$ as an oracle, outputs a copy $\Phi^\B$ of $\A$.
It is easy to see that if $\A$ is effectively interpretable in $\B$,  we can use the interpretation to build a Turing functional giving a Medvedev reduction from $\B$ to $\A$.
Kalimullin \cite{Kal09} showed that this implication cannot be reversed.
In this paper we consider a strengthening of Medvedev reducibility that is equivalent to effective interpretability. 
This strengthening comes from asking the Turing functional $\Phi$ to preserve isomorphisms in the following sense.
Given an isomorphism between two copies of $\B$, we want an effective way to compute an isomorphism between the two copies of $\A$ that we get by applying $\Phi$. 
We will define this more precisely using the language of category theory.

Throughout the paper, 
we write $\Iso{\A}$ for the isomorphism class of a countably infinite structure $\A$:
$$ \Iso{\A} = \{\Ahat~:~\Ahat\cong\A~\&~\text{dom}(\Ahat)=\omega\}.$$
We will regard $\Iso{\A}$ as a category, with the copies of the structures as its objects and the isomorphisms among them as its morphisms.

\begin{definition}\label{def:functors}
By a {\em functor from $\A$ to $\B$} we mean a functor from $\Iso{\A}$ to  $\Iso{\B}$, that is, a map $F$ that assigns to each copy $\Ahat$ in $\Iso{A}$ a structure $F(\Ahat)$ in $\Iso{\B}$, and assigns to each morphism $f\colon \Ahat \to \Atilde$ in $\Iso{\A}$ a morphism $F(f) \colon F(\Ahat) \to F(\Atilde)$ in $\Iso{\B}$ so that the  two properties hold below:
	\begin{enumerate}
		\item[(N1)]	$F(\id_{\Ahat}) = \id_{F(\Ahat)}$ for every $\Ahat \in \Iso{\A}$, and
		\item[(N2)] $F(f \circ g) = F(f) \circ F(g)$ for all morphisms $f,g$ in $\Iso{\A}$.
	\end{enumerate}

 A  functor $F\colon \Iso{\A} \rightarrow \Iso{\B}$ is {\em computable} if
there exist two computable operators $\Phi$ and $\Phi_*$ such that
\begin{enumerate}
\item[(C1)] for every $\Ahat \in \Iso{\A}$, $\Phi^{D(\Ahat)}$ is the atomic diagram of $F(\Ahat) \in \Iso{\B}$;
\item[(C2)] for every morphism $f:\Ahat\to\Atilde$ in $\Iso{\A}$,
$\Phi_*^{D(\Ahat)\oplus f\oplus D(\Atilde)}  = F(f). $
\end{enumerate}
Recall that $D(\Ahat)$ denotes the atomic diagram of $\Ahat$.
We will often identify a computable functor with the pair $(\Phi, \Phi_*)$ of Turing operators witnessing its computability. 
\end{definition}

Notice that $\Phi$, without $\Phi_*$, gives a Medvedev reduction from $\Iso{\A}$ to $\Iso{\B}$.
From the examples in the literature of Medvedev reducibilities, some turn out to be effective functors,
but not all.

Our first main result connects computable functors and effective interpretability.

\begin{theorem}
\label{thm:basicequiv}
Let $\A$ and $\B$ be countable structures. Then
$\A$ is effectively interpretable in $\B$ if and only if
there exists a computable functor from $\B$ to $\A$.
\end{theorem}

We prove Theorem~\ref{thm:basicequiv} in Section~\ref{sec:effective}. It is well-known in model theory that an elementary first-order interpretation of one structure in another gives rise to a functor. One can find a treatment of this fact in the book by Hodges \cite[pp. 216--218]{Hodges93}. The corresponding direction in Theorem~\ref{thm:basicequiv}---from left to right---is rather straightforward, and the only new thing is to consider the effectiveness of the functor. The interesting direction is to build an interpretation out of a functor.

Our proof of Theorem~\ref{thm:basicequiv} not only shows the existence of such an interpretation, but actually it builds a correspondence between functors and interpretations.
This last observation, which we will discuss in Proposition \ref{prop:kuku} and Section \ref{se:uniqueness}, is quite important.  For instance, when $\A$ has a computable copy Theorem~\ref{thm:basicequiv} is trivial and Proposition \ref{prop:kuku} is still meaningful: in this case we always have an effective interpretation of $\A$ into $\B$ which ignores the structure in $\B$, and also a functor from $\B$ to $\A$ that always outputs the same computable copy of $\A$ and the identity isomorphism on it without consulting the oracle.

Let us now explain how is that Proposition \ref{prop:kuku} extends Theorem~\ref{thm:basicequiv}.
Suppose we have a computable functor $F\colon \Iso{\B} \rightarrow \Iso{\A}$ whose effectiveness is witnessed by $(\Phi, \Phi_*)$. 
The backward direction of Theorem~\ref{thm:basicequiv} says that $\A$ must be effectively interpretable in $\B$. 
Applying the forward direction of Theorem~\ref{thm:basicequiv} to this effective interpretation, we get a computable functor based on this interpretation, denoting this new functor by $\mathcal{I}^F$ (here $\mathcal{I}$ stands for `interpretation'). 
We will show that these functors are isomorphic even in an effective way. The appropriate notion of equivalence is the following.

\begin{definition}\label{def:effeq}
A functor $F\colon \Iso{\B} \rightarrow \Iso{\A}$ is {\em effectively naturally isomorphic} (or just {\em effectively isomorphic}) to a functor $G\colon \Iso{\B} \rightarrow \Iso{\A}$   if there is a computable Turing functional $\Lambda$ such that for every $\Btilde \in \Iso{\B}$, $\Lambda^{\Btilde}$ is an isomorphism from $F(\Btilde)$ to $G(\Btilde)$, and the following diagram commutes for every $\Btilde, \Bhat \in \Iso{\B}$ and every morphism $h\colon \Btilde \to \Bhat$:
\[
\xymatrix{
F(\Btilde)\ar[d]_{F(h)}\ar[r]^{\Lambda^{\Btilde}} &      G(\Btilde)\ar[d]^{G(h)}   \\
F(\Bhat)\ar[r]_{\Lambda^{\Bhat}}    & G(\Bhat)
}\]
\end{definition}

\begin{proposition}\label{prop:kuku}
Let $F\colon \Iso{\B} \to \Iso{\A}$ be a computable functor. Then $F$ and $\mathcal{I}^F$ (defined above) are effectively isomorphic.
\end{proposition}

We prove Proposition \ref{prop:kuku} in Section  \ref{se:uniqueness}.

Suppose that $F$ and $G$ are functors, and $F \circ G$ and $G \circ F$ are effectively isomorphic to the identity. The witness to $G \circ F$ being effectively isomorphic to the identity functor is a Turing functional $\Lambda_{\A}$ which gives, for any $\Atilde \in \Iso{\A}$, a map $\Lambda_{\A}^{\Atilde} \colon \Atilde \to G(F(\Atilde))$. Thus, applying the functor $F$, we get a map $F(\Lambda_{\A}^{\Atilde}) \colon F(\Atilde) \to F(G(F(\Atilde)))$. There is also a map $\Lambda_{\B}^{F(\Atilde)} \colon F(\Atilde) \to F(G(F(\Atilde)))$ which is obtained from the Turing functional $\Lambda_{\B}$ which witnesses that $F \circ G$ is effectively isomorphic to the identity functor. If these two maps $F(\Atilde) \to F(G(F(\Atilde)))$ agree for every $\Atilde \in \Iso{\A}$, and similarly with the roles of $\A$ and $\B$ switched, then we say that $F$ and $G$ are \emph{pseudo-inverses}.

\begin{definition} Two structures $\A$ and $\B$ with domain $\omega$ are \emph{computably} \emph{bi-transformable} if there exist computable functors $F\colon \Iso{\A} \rightarrow \Iso{\B}$ and $G\colon \Iso{\B} \rightarrow \Iso{\A}$ which are pseudo-inverses.
\end{definition}

\begin{theorem}
\label{thm:bi-interpretability}
Let $\A$ and $\B$ be countable structures. Then $\A$ and $\B$ are effectively bi-interpretable  iff $\A$ and $\B$ are computably bi-transformable.
\end{theorem}

 We prove Theorem~\ref{thm:bi-interpretability} in Section~\ref{sec:gener}.

\subsection*{Effective transformations of classes}
There has been much work in the last few decades analyzing which classes of structures can be reduced to others, and which are universal in the sense that the class of all structures reduces to them.
The meaning of ``reduces'' has varied.
The intuition is that one class reduces to another if every structure in the first class can be somehow encoded by a structure in the second class, and usually we want the encoding structure to have similar complexity as the structure being coded.
For instance, a class is {\em universal for degree spectra} if every degree spectrum realized by some structure is realized by a structure in the class.
The most celebrated paper in this direction was written by Hirschfeldt, Khoussainov, Shore and Slinko \cite{HKSS}.
They defined what it means for a class to be complete with respect to degree spectra of nontrivial structures, effective dimensions, expansion by constants, and degree spectra of relations.
Then they showed that  undirected graphs,  partial orderings,  lattices,  integral domains of arbitrary characteristic (and in particular rings), commutative semigroups,  and 2-step nilpotent groups are all complete in these sense.
Their definition is rather cumbersome and does not seem to be equivalent to our definitions below, but the definitions appear rather close in spirit.  

Our intention is to apply the proofs of Theorems
\ref{thm:basicequiv} and \ref{thm:bi-interpretability} to the situation in which one class $\CC$
of countable structures is effectively interpretable in another
class $\DD$.

In what will follow, a \emph{class} is  a category of countable structures upon the domain $\omega$ and morphisms are permutations of $\omega$ that induce isomorphisms, and we also assume our classes are closed under such isomorphisms.  (That is, if $\A$ and $\B$ are objects in the class,
every isomorphism between them is a morphism in the class.)
We can extend the definition of a computable functor to arbitrary classes (not necessarily of the form $\Iso{\A}$) by simply allowing the oracles of $\Phi$ and $\Phi_*$ to range over the objects and morphisms of an arbitrary class.

\begin{definition}
Say that a class $\CC$ is \emph{uniformly transformally reducible} to a
class $\DD$ there exist a subclass $\DD' $ of $\DD$ 
and computable functors $F\colon \CC \rightarrow \DD'$, $G\colon \DD' \rightarrow \CC$  such that $F$ and $G$ are pseudo-inverses.

\end{definition}

The syntactical counterpart of the above definition is:

\begin{definition}[\cite{MonICM}]
Say that a class $\CC$ is \emph{reducible via effective bi-interpretability} to a
class $\DD$ if for every $\C \in \CC$ there is a $\D \in \DD$ such that $\C$ and $\D$ are effectively bi-interpretable and furthermore the formulae defining the interpretations and the isomorphisms do not depend on the concrete choice of $\C$ or $\D$. 

\end{definition}

We have:

\begin{theorem}  A class $\CC$ is reducible via effective bi-interpretability to a
class $\DD$ iff $\CC$ is uniformly transformally reducible to a
class $\DD$.
\end{theorem}

\begin{proof}
The proof of Theorem~\ref{thm:bi-interpretability}
is uniform in both directions.
\end{proof}

Using the interpretations defined by Hirschfeldt, Khoussainov, Shore and Slinko \cite{HKSS02}, we get the following:
 undirected graphs, partial orderings, and lattices are {\em on top (or universal) for effective bi-interpretability} (see \cite[Section 5.2]{MonICM}).
If we add a finite set of constants to the languages of integral domains, commutative semigroups, or 2-step nilpotent groups, they become on top for effective bi-interpretability too.
A recent result by J. Park, B. Poonen, H. Schoutens, A. Shlapentokh, and one of us \cite{MPPSS} shows that fields are also universal for effective bi-interpretability.


\section{Proof of Theorem~\ref{thm:basicequiv}}
\label{sec:effective}
We split the proof into two propositions, one proposition for each direction of Theorem~\ref{thm:basicequiv}. We start by quickly disposing of the easy direction.

\begin{proposition}\label{pr1} If $\A$ is effectively  interpretable in $\B$, then there exists a computable functor from $\Iso{\B}$ to $\Iso{\A}$.
\end{proposition}
\begin{proof}
Suppose that $\A$ is interpreted in $\B$ via $\Dom{\B}{\A}$, $\sim$, and $\langle R_i \rangle_{i \in \omega}$ as in Definition \ref{def: eff int}.
Given $\Btilde \in \Iso{\B}$, we first define $\Atilde = F(\Btilde)$ upon the domain $\omega$ as follows.
Notice that since the sequence of relations $\Dom{\B}{\A}$, $\sim$, and $\langle R_i \rangle_{i \in \omega}$ is $\Dec_1$ definable in $\B$, the respective interpretations in $\Btilde$ are uniformly computable from the open diagram $D(\Btilde)$ of $\Btilde$.
Since $\Btilde$ has domain $\om$, we have that $\Dom{\Btilde}{\A}\subseteq \om^{<\om}$ and using a fixed enumeration of $\om^{<\om}$ we get a bijection $\tilde{\tau}$:
\[
\tilde{\tau}\colon \om \to \Dom{\Btilde}{\A} / \sim.
\]
Note that $\tauti$ is uniformly computable from $D(\Btilde)$.
Using $\tauti$, we define relations $P_i$ on $\omega$ via the pull-back from $(\Dom{\Btilde}{\A}/\sim; R_0^{\Btilde},R_1^{\Btilde},...)$ along $\tilde{\tau}$, and let the resulting structure be $F(\Btilde)=\Atilde$. 

Also, given an isomorphism $f\colon \Btilde \rightarrow \Bhat$, we need to define an isomorphism $F(f)\colon F(\Btilde)\to F(\Bhat)$.
Using the respective bijections $\tilde{\tau}$ and $\hat{\tau}$  as above, 
and extending $f$ to the domain $\Btilde^{<\omega}$ in the obvious way,
we define 
\[
F(f) = \hat{\tau}^{-1} \circ f \circ \tilde{\tau}   \colon \Atilde \rightarrow \Ahat
\]
It is straightforward to check that the above definition of $F$ gives a functor from $\Iso{\B}$ to $\Iso{\A}$.
\end{proof}

We now move on to the more interesting direction.

\begin{proposition} \label{main propo}
Suppose there exists a computable functor from $\Iso{\B}$ to $\Iso{\A}$. Then $\A$ is effectively  interpretable in $\B$.
\end{proposition}

\begin{proof}

Let  $F = (\Phi,\Phi_*)$ be a computable functor from $\Iso{\B}$
into $\Iso{\A}$. We will produce $\Sigma^c_1$-formulas
for an effective interpretation of $\A$ in $\B$.
We begin by introducing some notation and conventions. We will then  define  $\Dom{\B}{\A}$ and $\sim$ formally and prove several useful lemmas about them. 
After that we define $R_i$ and show that our definitions suffice.

\bigskip

\noindent \textbf{Notations and conventions.} We identify a function $f\colon \omega \rightarrow \omega$ with its graph, using $\lambda$ to denote the identity function on $\omega$.
If $\bar{x} = (x_0, \ldots, x_n)$ and $\sigma$ is a permutation of $\{0,\ldots,n\}$, then $(\bar{x})_{\sigma}$ is the tuple $(x_{\sigma(0)},\ldots,x_{\sigma(n)})$. 
For $\bar{b} \in \B$ we view $\bbar$ as a partial map which takes the tuple $(0,\ldots,|\bar{b}|-1)$ to $(b_0,b_1,...,b_{|\bar{b}|-1})$. 
Viewing $\xbar$ as a partial map, note that $(\xbar)_\sigma = \xbar\circ \sigma$.
  
If $f$ is a map from $\omega$ to the domain of $\B$, then we can ``pull back'' the structure on $\B$ along $f$ to get a structure $\B_{f}$ on $\omega$ such that $f\colon \B_{f} \to \B$ is an isomorphism. 
 Given a tuple $\bar{b} \in \B$ and $f \supset {\bar{b}}$, we write $D(\bar{b})$ to denote the partial atomic diagram of $(0, 1, \ldots, |b|-1)$ in $\B_f$ that mentions only the first $|\bar{b}|$-many relations.
This partial atomic diagram will be typically identified with the finite binary string that, under some fixed G$\rm\ddot{o}$del numbering of the atomic formulas, encodes $D(\bar{b})$. 
Thus, $D(\bbar)$ is an initial segment of the atomic diagram $D(\B_f)$ of $\B_f$.
Furthermore, $D(\B_f)=\bigcup_{n\in\om} D(f\upto n)$.
Note that $D(\bar{b})$ does not really depend on a particular choice of $f$ as long as $f \supset {\bar{b}}$; it only depends on what atomic formulas hold of $\bbar$.

 Finally, for finite tuples $\bar{b}$ and $\bar{c}$, we write $\bar{c}\setminus \bar{b}$ for the set of elements that occur in $\bar{c}$ but not in $\bar{b}$.

\bigskip

\noindent \textbf{Definitions of $\mathbf{\Dom{\B}{\A}}$ and $\sim$.} Recall that $\B^{< \omega} \times \omega$ can be easily coded by elements of $\B^{<\omega}$.  
 We  define the domain $\Dom{\B}{\A}$ and the equivalence relation $\sim$ upon that domain as follows:

\smallskip

\begin{enumerate}
 \item[$\mathbf{\Dom{\B}{\A}}$:] Define $\Dom{\B}{\A}$ to be the set of pairs $(\bar{b},i) \in \B^{< \omega} \times \omega$ such that
\[ \Phi_*^{D(\bar{b}) \oplus \lambda\upto {|\bar{b}|} \oplus D(\bar{b})}(i) \downarrow = i. \]

\item[$\mathbf{\sim}$:]  For  $(\bar{b},i), (\bar{c},j) \in \Dom{\B}{\A}$, let $(\bar{b},i) \sim (\bar{c},j)$ if there exists a finite tuple $\bar{d}$ that does not mention elements from $\bar{b}$ and $\bar{c}$, such that if we let $\bar{c}'$ and $\bar{b}'$  list the elements in  $\bar{c}\setminus \bar{b}$ and $\bar{b}\setminus \bar{c}$ respectively, and let $\sigma$ be the finite permutation with $(\bar{b}\bar{c}' \bar{d}) = (\bar{c}\bar{b}' \bar{d})_\sigma$, then

$$ \Phi_*^{D(\bar{b}\bar{c}'\bar{d}) \oplus \sigma \oplus D(\bar{c}\bar{b}'\bar{d})} (i) \downarrow = j \text{ and } \Phi_*^{D(\bar{c}\bar{b}' \bar{d}) \oplus \sigma^{-1} \oplus D(\bar{b}\bar{c}'\bar{d})} (j) \downarrow = i.$$

\end{enumerate}

Intuitively, given $\bbar\subseteq f$, we have that $D(\bbar)\subseteq D(\B_f)$, and hence $\Phi^{D(\bbar)}$ is a finite initial segment of $\Phi^{D(\B_f)}$ which is isomorphic to $\A$.
The idea is that $(\bbar,i)$ will represent the $i$th element in the presentation $\Phi^{D(\B_f)}$ of $\A$.
Of course, there are many possible $f\colon\om\to\B$ extending $\bbar$, and the element $i$ on the different presentations $\Phi^{D(\B_f)}$  may correspond to different elements of $\A$.
As we will see later, the condition we are imposing to have $(\bar{b},i)\in \Dom{\B}{\A}$ will guarantee that this $i$th element always corresponds to the same element in $\A$.

The intuition behind $\sim$ is that the partial diagrams $D(\bar{b}\bar{c}'\bar{d})$, $D(\bar{c}\bar{b}'\bar{d})$, and the isomorphism
between them are enough information for $\Phi_*$ to recognize that
the element $i$ of $\Phi^{\B_{f}}$ should be paired with the element
$j$ of $\Phi^{\B_{g}}$ for any $f \supset {\bar{b}\bar{c}'\bar{d}}$ and $g \supset {\bar{c}\bar{b}'\bar{d}}$.
We note that $\sigma\subseteq g^{-1}\circ f\colon \B_f\to\B_g$.

\medskip

\noindent \textbf{Properties of $\mathbf{\Dom{\B}{\A}}$ and $\sim$.} Before we proceed, we verify that our definitions of $\Dom{\B}{\A}$ and $\sim$ satisfy the nice properties that one would expect from the ``right'' definitions of $\Dom{\B}{\A}$ and $\sim$. 

\begin{lemma}\label{lem:1} The set
$\Dom{\B}{\A}$  and its complement are both definable  in the language of $\B$ by $\Sigma^c_1$-formulas without parameters.
\end{lemma}
\begin{proof}
One can simply observe that $\Dom{\B}{\A}$ is u.r.i.\ computable in $\B$, and hence $\Dec_1$-definable without parameters.  However, let us also include a more syntactical proof to give the reader a better idea of what is going on.
We can enumerate the diagrams $D(\bbar)$ for which $\Phi_*^{D(\bbar)\oplus
\lambda\res |\bbar|\oplus D(\bbar)}(i)$ converges and is equal to $i$, and we can also compute the diagrams for which the computation diverges or does not equal~$i$. 
Each of these finite partial diagrams corresponds to a quantifier-free formula about~$\bar{b}$.
(Notice that here ``divergence'' does not mean that the computation runs forever;
indeed, $\Phi_*^{D(\B)\oplus\lambda\oplus D(\B)}$ must be total.
Rather, we say that the computation diverges on an input if it demands information
about $D(\B)$ or about $\lambda$ that the finite oracle does not include,
in which case we will recognize that the computation has diverged in this sense.  If it fails to diverge
in this sense, then it must in fact halt.)  Then $\Dom{\B}{\A}$ is defined by the computable disjunction of those formulas corresponding to diagrams where the computation converges and is equal to $i$, and its complement is defined by the disjunction of the other formulas (i.e., where the computation diverges or is not equal to $i$).

To ensure that the same computable disjunction works for every structure $\Btilde\in\Iso{\B}$,
we include in the disjunction \emph{every} finite string $\delta$ for which
$\Phi_*^{\delta\oplus\lambda\res k\oplus\delta}(i)\converges=i$
(where $k$ is the length of the tuple about which $\delta$ could be a fragment
of an atomic diagram).  After all, the functional $\Phi_*$ has no particular idea
which copy of $\B$ it has for its oracle.  Likewise, the computable disjunction defining
the complement of $\Dom{\B}{\A}$ includes every finite $\delta$ for which $\Phi_*^{\delta(i)\oplus\lambda\res k\oplus\delta}(i)$ either converges
to a value $\neq i$, or diverges by demanding more information than $\delta$ or
$\lambda\res k$ contains (as described above).  These are both $\Sigma_1^c$ disjunctions:
there may exist certain $\delta$ for which $\Phi_*^{\delta\oplus\lambda\res k\oplus\delta}(i)$
neither converges nor
demands too much information, but because $(\Phi,\Phi_*)$ is assumed to be a
computable functor, such a $\delta$ cannot be an initial segment of the atomic diagram
of any copy of $\B$.
\end{proof}

\begin{lemma}\label{lem:2}
The binary relation $\sim$ and its complement are both definable  in the language of $\B$ by $\Sigma^c_1$-formulae without parameters.
\end{lemma}

\begin{proof}

It is clear that $\sim$ has a $\Sigma^c_1$-definition (the same argument as in Lemma~\ref{lem:1}). 
We claim that the complement of $\sim$ also has a $\Sigma^c_1$-definition, but this has a  more complicated proof. 
Aiming for a definition of the complement of $\sim$ (and slightly abusing notations), 
we define a new binary relation $\nsim$ as follows. Let  $(\bar{b},i) \nsim (\bar{c},j)$ if there exist  $\dbar$ as in the definition of $\sim$ except that 
\[  
\Phi_*^{D(\bar{b}\bar{c}'\bar{d}) \oplus \sigma \oplus D(\bar{c}\bar{b}'\bar{d})} (i) \downarrow \not= j \text{ or } \Phi_*^{D(\bar{c}\bar{b}' \bar{d}) \oplus \sigma^{-1} \oplus D(\bar{b}\bar{c}'\bar{d})} (j) \downarrow \not= i.
\]
If we show that $\nsim$ is equal to the complement of $\sim$ (as the notation suggests) then we are done, since $\nsim$ clearly has a $\Sigma^c_1$-definition.
Thus, it is sufficient to prove that
for $(\bar{b},i), (\bar{c},j) \in \Dom{\B}{\A}$, we have exactly one of $(\bar{b},i) \nsim (\bar{c},j)$ and $(\bar{b},i) \sim (\bar{c},j)$.

First, we will show that at least one of $(\bar{b},i) \sim (\bar{c},j)$ or $(\bar{b},i) \nsim (\bar{c},j)$ holds. 
Let $\bar{b}'$ and $\bar{c}'$ be tuples consisting of the elements in $\bar{b}$ but not in $\bar{c}$, and in $\bar{c}$ but not in $\bar{b}$, respectively.
Let $\sigma$ be the map that matches the elements of $\bar{b}, \bar{c'}$ with their natural copies in $\cbar, \bbar'$, so that $(\bbar,\cbar')=(\cbar,\bbar')_\sigma$. 
Let $f, g\colon \om\to\B$ be bijections extending $\bar{b}, \bar{c'}$ and $\cbar, \bbar'$ respectively, and which coincide on all inputs $i\geq |\bar{b}, \bar{c'}|=|\cbar, \bbar'|$.
Thus, $h = g^{-1}\circ f$ is a permutation of $\om$ extending $\si$ which is constant on all inputs $i\geq |\si|$.
Recall that  $\B_f$ and $\B_g$ are the structures in $\Iso{\B}$ that we get by pulling back $f$ and $g$.
Observe that $h$ is an isomorphism from $\B_f$ to $\B_g$. 
Thus, by the choice of $\Phi_*$, we must have
\[ 
\Phi_*^{D(\B_f) \oplus h \oplus D(\B_g)} (i) \downarrow = j' \text{ and } \Phi_*^{D(\B_g) \oplus h^{-1} \oplus D(\B_f)} (j) \downarrow = i' 
\]
for some $i',j' \in \omega$.
Let us now consider an initial segment of these oracles where these computations still converge.
That is, for some $\dbar$ with $\bbar\cbar'\dbar\subset f$ and $\cbar\bbar'\dbar\subset g$, and for $\si'\supseteq\si$ so that $(\bbar\cbar'\dbar)=(\cbar\bbar'\dbar)_{\si'}$, we have 
\[ 
\Phi_*^{D(\bbar\cbar'\dbar) \oplus \sigma' \oplus D(\cbar\bbar'\dbar)} (i) \downarrow = j' \text{ and } \Phi_*^{D(\cbar\bbar'\dbar) \oplus {\sigma'}^{-1} \oplus D(\bbar\cbar'\dbar)} (j) \downarrow = i'. \]
If $i = i'$ and $j = j'$, 
we get $(\bar{b},i) \sim (\bar{c},j)$, and if either $i\neq i'$ or $j\neq j'$, we get $(\bar{b},i) \nsim (\bar{c},j)$.

\medskip{}

Second, we show that $(\bar{b},i) \sim (\bar{c},j)$ and $(\bar{b},i) \nsim (\bar{c},j)$ do not hold at the same time.  Suppose the contrary.
Let $\si$ and $\bar{d}_1$ witness that $(\bar{b},i) \sim (\bar{c},j)$, and $\tau$ and $\bar{d}_2$ witness that $(\bar{b},i) \nsim (\bar{c},j)$.  
Without loss of generality, we may assume
\[ \Phi_*^{D(\bar{b}\bar{c}'\bar{d}_1) \oplus \sigma \oplus D(\bar{c}\bar{b}'\bar{d}_1)}(i) = j \text{ but } \Phi_*^{D(\bar{b}\bar{c}'\bar{d}_2) \oplus \tau \oplus D(\bar{c}\bar{b}'\bar{d}_2)}(i) \neq j .\]

Choose bijective maps from $\om$ to $\B$ such that 
\[ 
f_1 \supset {\bar{b}\bar{c}'\bar{d}_1},
\quad
g_1 \supset {\bar{c}\bar{b}'\bar{d}_1},
\quad
f_2 \supset {\bar{b}\bar{c}'\bar{d}_2},
\quad
g_2 \supset {\bar{c}\bar{b}'\bar{d}_2}.
\] 
Then we have isomorphisms
\[ \xymatrix{
 			&   &   \B			\\
\B_{f_1} \ar[r]^{g_1^{-1} \circ f_1}   \ar@/_1pc/[rrrr]_{f_2^{-1} \circ f_1}   \ar@/^1pc/[urr]^{f_1}  &  \B_{g_1} \ar[rr]^{g_2^{-1} \circ g_1}   \ar@/^.5pc/[ur]^{g_1}     & &   \B_{g_2} \ar[r]^{f_2^{-1} \circ g_2}   \ar@/_.5pc/[ul]_{g_2}  & \B_{f_2}     \ar@/_1pc/[ull]_{f_2}.    \\
{\bar{b}\bar{c}'\bar{d}_1} &    \ar[l]^{(\cdot)_\si}   	{\bar{c}\bar{b}'\bar{d}_1}	 &  & {\bar{c}\bar{b}'\bar{d}_2}	  \ar[r]_{(\cdot)_\tau}	& {\bar{b}\bar{c}'\bar{d}_2}  
}\]
Since $F$ is a functor, we have
\[
 F(f_2^{-1} \circ f_1) = F(f_2^{-1} \circ g_2) \circ F(g_2^{-1} \circ g_1) \circ F(g_1^{-1} \circ f_1). 
 \]
Note that $g_2^{-1} \circ g_1 \supset \lambda \upto {|\bar{c}|}$ and $f_1^{-1} \circ f_2 \supset \lambda \upto {|\bar{b}|}$. 
Now, on the one hand, since $(\bar{b},i)$ and $(\bar{c},j)$ are in $\Dom{\B}{\A}$, we have:
\[ 
F(f_1^{-1} \circ f_2)(i) = \Phi_*^{\B_{f_1} \oplus (f_2^{-1} \circ f_1) \oplus \B_{f_2}}(i) = \Phi_*^{D(\bar{b}) \oplus \lambda\upto {|\bar{b}|} \oplus D(\bar{b})}(i) = i,
 \]
\[ 
F(g_2^{-1} \circ g_1)(j) = \Phi_*^{\B_{g_1} \circ (g_2^{-1} \circ g_1) \oplus \B_{g_2}}(j) = \Phi_*^{D(\bar{c}) \circ \lambda\upto {|\bar{c|}} \oplus D(\bar{c})}(j) = j .
\]
On the other hand, since $g_1^{-1} \circ f_1 \supset \sigma$ and $g_2^{-1} \circ f_2 \supset \tau$ we have:
\[ 
F(g_1^{-1} \circ f_1)(i) = \Phi_*^{\B_{f_1} \oplus (g_1^{-1} \circ f_1) \oplus \B_{g_1}}(i) = \Phi_*^{D(\bar{b}\bar{c}'\bar{d}_1) \oplus \sigma \oplus D(\bar{c}\bar{b}'\bar{d}_1)}(i) = j,
\]
\[ 
F(g_2^{-1} \circ f_2)(i) = \Phi_*^{\B_{f_2} \circ (g_2^{-1} \circ f_2) \oplus \B_{g_2}}(i) =  \Phi_*^{D(\bar{b}\bar{c}'\bar{d}_2) \oplus \tau \oplus D(\bar{c}\bar{b}'\bar{d}_2)}(i) \neq i .
\]
Composing the latter three equation lines, we get that $F(f_1^{-1} \circ f_2)(i)\neq j$,
contradicting the first line.
\end{proof}

\begin{lemma}\label{lem:eqr}
On its domain, $\Dom{\B}{\A}$, the relation $\sim$ is an equivalence relation.
\end{lemma}
\begin{proof}
It is evident that $\sim$ is symmetric (use $\sigma^{-1}$) and reflexive (since $(b,i) \in  \Dom{\B}{\A}$).
We show that $\sim$ is transitive. Suppose that $(\bar{a},i)$, $(\bar{b},j)$, and $(\bar{c},k)$ are in $\Dom{\B}{\A}$ and are such that $(\bar{a},i) \sim (\bar{b},j)$ and $(\bar{b},j) \sim (\bar{c},k)$.

Let $\bar{b}', \bar{a}', \bar{d}', \sigma$ witness $(\bar{a},i) \sim (\bar{b},j)$, and let $\bar{c}'', \bar{b}'', \bar{d}'', \tau$ witness $(\bar{b},j) \sim (\bar{c},k)$ (see the definition of $\sim$).
Let $\bar{c}'''$ be a string listing $\bar{c}\setminus \bar{a}$, and $\bar{a}'''$ be a string listing $\bar{a} \setminus \bar{c}$.
Choose bijections from $\om$ to $\B$ as follows:
\begin{align*}
f_1 & \supset {\bar{a}\bar{b}'\bar{d}'} & g_1 &\supset {\bar{b}\bar{c}''\bar{d}''}  & h_1 & \supset \bar{a}\bar{c}'''  \\
f_2 & \supset {\bar{b}\bar{a}'\bar{d}'} & g_2 &\supset {\bar{c} \bar{b}'' \bar{d}''} & h_2 & \supset \bar{c}\bar{a}'''
\end{align*}
where $h_1$ and $h_2$ agree outside the initial segment of length $|\bar{a}|+|\bar{c}'''| = |\bar{c}+\bar{a}'''|$.

\[ \xymatrix{
 		&	&   &   \B		& & 	\\
\B_{h_1} \ar[r]^{f_1^{-1} \circ h_1}   \ar@/_1.4pc/[rrrrrr]^{h_2^{-1} \circ h_1}   \ar@/^1.7pc/[urrr]^{h_1}       &  \B_{f_1} \ar[r]^{f_2^{-1} \circ f_1}   \ar@/^.8pc/[urr]^{f_1}          &  \B_{f_2} \ar[rr]^{g_1^{-1} \circ f_2}   \ar@/^.5pc/[ur]_{f_2}     &           &   \B_{g_1} \ar[r]^{g_2^{-1} \circ g_1}   \ar@/_.5pc/[ul]^{g_1}                        &   \B_{g_2} \ar[r]^{h_2^{-1} \circ g_2}   \ar@/_.8pc/[ull]_{g_2}          &         \B_{h_2}     \ar@/_1.7pc/[ulll]_{h_2}.    \\
{\bar{a}\bar{c}'''}   \ar@{<-}@/^1.4pc/[rrrrrr]_{(\cdot)_\rho} 	&{\bar{a}\bar{b}'\bar{d}'}	\ar@{<-}[r]_{(\cdot)_\si}  & {\bar{b}\bar{a}'\bar{d}'}	   &	  & {\bar{b}\bar{c}''\bar{d}''}		  & {\bar{c}\bar{b}''\bar{d}''}	\ar@{<-}[l]^{(\cdot)_{\tau}}	& {\bar{c}\bar{a}'''}   
}\]

Since ${F}$ is a functor, we have 
\[
F( h^{-1}_2 \circ h_1) = F(h_2^{-1}\circ g_2)\circ F(g^{-1}_2\circ g_1)\circ F(g^{-1}_1\circ f_2) \circ F(f^{-1}_2\circ f_1) \circ F(f_1^{-1}\circ h_1).
\]
Note also that $(\bar{b}, j) \in \Dom{\B}{\A}$ and $g^{-1}_1\circ f_2 \supset \lambda \upto {|\bar{b}|}$  imply  $F(g^{-1}_1\circ f_2)(j) = j$.
Similarly, $F(f_1^{-1}\circ h_1)(i) = i$ and $F(h_2^{-1}\circ g_2)(k) = k$.
We also have $F(f^{-1}_2\circ f_1)(i) = j$ and $F(g^{-1}_2\circ g_1)(j) = k$ by the choice of $f_1, f_2, g_1$ and $g_2$. Thus, 
$F(h^{-1}_2\circ h_1)(i) = k$
which must be witnessed by $\Phi_{*}^{ \B_{h_1}\oplus (h_2^{-1} \circ h_1) \oplus \B_{h_2}}(i) = k$. A symmetric argument shows 
$\Phi_*^{\B_{h_2}\oplus (h_1^{-1} \circ h_2) \oplus \B_{h_1}}(k) = i$. Now recall that $h_1$ and $h_2$ agree outside the initial segment of length $|\bar{a}|+|\bar{c}'''| = |\bar{c}|+|\bar{a}'''|$. Thus, for some long enough $\bar{e}$ and for  $\rho \subset h_2^{-1} \circ h_1$, the permutation mapping $\bar{a}\bar{c}''' \bar{e}$ to $\bar{a}''', \bar{c}''', \bar{e}$ we get a witness  for $(a,i) \sim (c,k)$. 
\end{proof}

The following two lemmas will be useful later. Their proofs are not difficult and can be skipped in a  first reading of the paper.

\begin{lemma}\label{lem:initial-segment}
For $(\bar{b},i) \in \Dom{\B}{\A}$, there is an initial segment $\bar{c} = \B\upto n$ of $\B$ and $j \in \omega$ such that $(\bar{b},i) \sim (\bar{c},j)$.
\end{lemma}

By $\B\upto n$ we mean the tuple that corresponds to $(0,1,....,n-1)$ in this given presentation $\B$.

\begin{proof}
Let $n$ be sufficiently large that $\bar{b} \in \B\upto n$. Let $\sigma$ be a permutation of $\{0,\ldots,n-1\}$ such that $\sigma(0,\ldots,|\bar{b}|-1) = \bar{b}$. Extend $\sigma$ to a permutation $f$ of $\omega$ by setting $f$ to be the identity on $\{n,n+1,\ldots\}$. Then let $j$ be such that
 $F(f)(i) = \Phi_*^{D(\B_f) \oplus f \oplus D(\B)}(i) = j. $
Since $F$ is a functor,
$ i=F(f^{-1})(j) = \Phi_*^{D(\B) \oplus f^{-1} \oplus D(\B_f)}(j). $
Let $m > n$ be such that these computations use only the first $m$ relation symbols and elements of $\B$. Let $\bar{c}' = f(|\bar{b}|,\ldots,m-1)$ and $\bar{c} = \B \upto m$. Then $\bar{b} \bar{c}'$ and $\bar{c}$ contain the same elements. Let $\tau = f\upto m$, so that $(\bar{b}\bar{c}') = (\bar{c})_{\tau}$. Then $(\bar{b},i) \sim (\bar{c},j)$ as witnessed by $\tau$.
\end{proof}

\begin{lemma}\label{lem:intial-segments-agreeing}
If  $(\bar{b},i)$ and $(\bar{c},j)$ are in $\Dom{\B}{\A}$, and $\bbar\subseteq \cbar$, then $(\bar{b},i) \sim (\bar{c},j)$ if and only if $i = j$.
\end{lemma}
\begin{proof}
Since $(\bar{b},i) \in \Dom{\B}{\A}$, we have
\[
\Phi_*^{D(\bar{b}) \oplus \lambda\upto {|\bbar|} \oplus D(\bar{b})}(i) = i.
\]
Let  $\bar{c}' = \cbar\sminus \bbar$ and let $\bar{d}$ and $\sigma\supset \lambda\upto {|\bbar|}$ witness that either $(\bar{b},i) \sim (\bar{c},j)$ or that $(\bar{b},i) \not\sim (\bar{c},j)$ as in the proof of Lemma \ref{lem:2}. 
Thus 
\[
\Phi_*^{D(\bar{b}\bar{c}'\bar{d}) \oplus \sigma \oplus D(\bar{c}\bar{d})}(i) = j', 
\]
for some $j'$, and $(\bar{b},i) \sim (\bar{c},j)$ hold if and only if $j=j'$.
But the oracle for this computation extends the oracle $D(\bar{b}) \oplus \lambda\upto {|\bbar|} \oplus D(\bar{b})$.
Therefore $j'=i$.
\end{proof}

\bigskip

\noindent \textbf{Defining the relations.} For each relation symbol $P_i$ of arity $p(i)$ in the language of $\A$ (recall that $p$ is a computable function), we define a relation $R_i$ on $\Dom{\B}{\A}$ as follows:

\begin{enumerate}
\item[$R_i$:] We let $(\bar{b}_1,k_1),\ldots,(\bar{b}_{p(i)},k_{p(i)})$ be in $R_i$ if there is a tuple $\bar{c}$ and $j_1,\ldots,j_{p(i)} \in \omega$ such that $(\bar{b}_s,k_s) \sim (\bar{c},j_s)$  for each $1 \leq s \leq p(i)$, and the atomic formula $P_i(j_1,\ldots,j_{p(i)})$ is true in $\Phi^{D(\bar{c})}$.

\end{enumerate}
We define a relation $Q_i$ the same way, except that $Q_i$ requires $P_i(j_1,\ldots,j_{a(i)})$
to be false in $\Phi^{D(\bar{c})}$.  (We will show $Q_i$ is the complement of $R_i$.)

Lemma~\ref{lem:2} combined with a standard argument (see, e.g., Lemma~\ref{lem:1}) imply that both $R_i$ and $Q_i$ are definable by a $\Sigma^c_1$ formula without parameters, and these formulae can be defined uniformly in $i$. 
Alternatively, it is not hard to see they are u.r.i.c.e.
The following lemma implies that $(R_i:i\in\om)$ is $\Dec_1$-definable without parameters. 

Fix $i$. We suppress $i$ in $R_i$, $Q_i$, and $p(i)$.

\begin{lemma}\label{lem:disjRQ}
$Q$ is the complement of $R$ in $\Dom{\B}{\A}$.
\end{lemma}
\begin{proof}
First, we need to show that each $(\bar{b}_1,i_1),\ldots,(\bar{b}_p,i_p)$ in $\Dom{\B}{\A}$ is either in $Q$ or in $R$. By Lemma \ref{lem:initial-segment} and Lemma~\ref{lem:intial-segments-agreeing}, for some sufficiently long initial segment $\bar{c}$ of the presentation $\B$, there are $j_1,\ldots,j_p$ such that $(\bar{b}_k,i_k) \sim (\bar{c},j_k)$ for each $1 \leq k \leq p$. Now $\Phi^{D(\B)}$ determines either that $(j_1,\ldots,j_p)$ is in $P$, or that it is not in $P$. By extending $\bar{c}$ to the use of this computation and using Lemma~\ref{lem:intial-segments-agreeing}, we get that $(\bar{b}_1,i_1),\ldots,(\bar{b}_p,i_p)$ is either in $Q$ or in $R$.

\medskip

We show that $(\bar{b}_1,i_1),\ldots,(\bar{b}_p,i_p)$ cannot be both in $Q$ and in $R$. Aiming for a contradiction, suppose  that there are $\bar{c}$ and $\bar{d}$, and $j_1,\ldots,j_p$ and $k_1,\ldots,k_p$, such that $(\bar{b}_m,i_m) \sim (\bar{c},j_m)$ and $(\bar{b}_m,i_m) \sim (\bar{d},k_m)$ for $1 \leq m \leq p$, and the atomic formula $P(j_1,\ldots,j_p)$ is in $\Phi^{D(\bar{c})}$, but $\neg P(k_1,\ldots,k_p)$ is not in $\Phi^{D(\bar{d})}$. 
 Note that by the transitivity of $\sim$, for each $m$ we have $(\bar{c},j_m) \sim (\bar{d},k_m)$.

Let $f \supset {\bar{c}}$ and $g \supset {\bar{d}}$ be permutations $\omega \to \B$. 
Then, since $\Phi^{D(\bar{c})}$ says that $P(j_1,\ldots,j_p)$ holds, and since $D(\bar{c})\subseteq D(\B_f)$, in $F(\B_f)$ the tuple $(j_1,\ldots,j_p)$ belongs to  $P^{F(\B_f)}$. 
Similarly, since $\Phi^{D(\bar{d})}$ says that $\neg P(k_1,\ldots,k_p)$, the tuple $(k_1,\ldots,k_p)$ is not in $P^{F(\B_g)}$.

The map $g^{-1} \circ f \colon \B_f \to \B_g$ is an isomorphism. 
With $(\bar{c},j_m) \sim (\bar{d},k_m)$ we must have $F(g^{-1} \circ f)(j_m)=k_m$,
since otherwise $(\bar{c},j_m)\not \sim (\bar{d},k_m)$ as in the proof of Lemma \ref{lem:2}.
So the isomorphism $F(g^{-1} \circ f) \colon F(\B_f) \to F(\B_g)$ maps $(j_1,\ldots,j_p)$ to $(k_1,\ldots,k_p)$, yielding a contradiction. 
\end{proof}
	
Thus,  for each relation symbol $P_i$ in the language of $\A$, we get a relation $R_i$ interpreting $P$ which is uniformly $\Delta^c_1$.
The corollary below follows from the proof of the previous lemma.

\begin{corollary}\label{cor:relations}
If $(\bar{b}_1,i_1),\ldots,(\bar{b}_p,i_p)$ and $(\bar{c}_1,j_1),\ldots,(\bar{c}_p,j_p)$ are all in $\Dom{\B}{\A}$, with $(\bar{b}_m,i_m) \sim (\bar{c}_m,j_m)$ for each $m$, then $(\bar{b}_1,i_1),\ldots,(\bar{b}_p,i_p)$ is in $R$ if and only if $(\bar{c}_1,j_1),\ldots,(\bar{c}_p,j_p)$ is in $R$.
\end{corollary}

\noindent \textbf{Defining an isomorphism.} We already know, from Lemma~\ref{lem:eqr}, that $\sim$ is an equivalence relation, and Corollary~\ref{cor:relations} says that $\sim$ agrees with our definition of $R_i$. Thus, $(\Dom{\B}{\A} / \sim; R_0/\sim,R_1/\sim,...)$ is a structure that can be viewed as a structure in the language of $\A$ (interpreting $P_i$ as $R_i/\sim$).
To finalize the proof, we need to define an isomorphism between 
\[
(\Dom{\B}{\A} / \sim; R_0/\sim,R_1/\sim,...)
\quad\mbox{and}\quad
\A= (A; P_0^\A,P_1^\A,\ldots).
\]
Using our fixed presentation $\B$, we define $\mathfrak{F}\colon A\to \Dom{\B}{\A}$ as follows:
Given $i\in \om= A$, let $\mathfrak{F}(i)= (\cbar,i)$ where $\cbar=\Bbar\upto n$ for the least $n\in\om$ such that $(\cbar,i)\in \Dom{\B}{\A}$.

\begin{lemma}\label{lem:isomorphism}
The function $\mathfrak{F} \colon\A\to \Dom{\B}{\A}$ defined above induces an isomorphism of 
 $(\Dom{\B}{\A} / \sim; R_0/\sim,R_1/\sim,...)$ onto  $(A; P_0^\A,P_1^\A,...).$
\end{lemma}
\begin{proof}
Lemma \ref{lem:intial-segments-agreeing} shows $\mathfrak{F}$ to be one-to-one.
Lemma~\ref{lem:initial-segment} shows it to be onto.
That it is an isomorphism follows directly from the definitions of $R_i$.
\end{proof}

This completes the proof of the proposition and thus of Theorem~\ref{thm:basicequiv}. \end{proof}

Abusing terminology, we will often refer to maps such as
$\mathfrak{F} \colon\A\to \Dom{\B}{\A}$ in Lemma \ref{lem:isomorphism}
as \emph{isomorphisms}, although in fact they only induce isomorphisms.
Likewise, a relation on $A\times \Do_\A^\B$ may be called an isomorphism
from $\A$ onto $\Do_\A^\B$
if it becomes one after modding out on the right by the equivalence $\sim$.
Finally, a composition of such ``isomorphisms'' may also be called
an isomorphism, as when we have maps between $\A$ and $\Do_\A^{\Do_B^\A}$.


\section{Effective uniqueness.}   \label{se:uniqueness}

 This section is devoted to a further analysis of Theorem~\ref{thm:basicequiv}. 
 We will prove Proposition \ref{prop:kuku}, which describes more explicitly what we actually get from the proof of Theorem~\ref{thm:basicequiv}.
Recall that Proposition \ref{prop:kuku} states that if $F\colon \Iso{\B} \to \Iso{\A}$ is a computable functor, then it is effectively isomorphic to $\I^F$, where $\I^F$ is the functor we get by transforming $F$ into an effective interpretation as in the proof of Proposition \ref{main propo} and then transforming it back into a computable functor using Proposition \ref{pr1}.

\begin{proof}[Proof of Proposition \ref{prop:kuku}]
For a presentation $\B$, set $\A=F(\B)$.  We will define 
\[
\Lambda^\B\colon F(\B) \to \I^F(\B).
\]
On the one hand, note that the map $\Fra \colon  F(\B) \to \Dom{\B}{\A}$ from Lemma \ref{lem:isomorphism} can be computed uniformly from a presentation of $\B$.
To be more explicit, we denote it by $\Fra^\B$. 
On the other hand, recall from the proof of Proposition \ref{pr1} that we build $\I^F(\B)$ out of the interpretation of $\A$ within $\B$ by pulling back through a bijection $\tau\colon \om\to \Dom{\B}{\A}$.
Let us call this bijection $\tau^\B$; it gives a well-defined
isomorphism from $\I^F(\B)$ to $\Dom{\B}{\A}/\sim$.
We define
\[
\Lambda^\B = (\tau^{\B})^{ -1} \circ \Fra^\B \colon F(\B)\to \I^F(\B).
\]
We need to show that $\Lambda$ is a natural isomorphism.
It is clear that $\Lambda(\B)$ is an isomorphism.
We must prove that, for all $\Btilde,\Bhat \in \Iso{\B}$ and all isomorphisms $h\colon \Btilde \to \Bhat$, the following diagram commutes.

\[
\xymatrix{
F(\Btilde)\ar[d]_{F(h)}  \ar[r]^{\Fra^{\Btilde}}    \ar@/^2pc/[rr]^{\Lambda^{\Btilde}}       & \Dom{\Btilde}{\A}\ar[d]_{h}        &     \mathcal{I}^{F}(\Btilde)\ar[d]^{\mathcal{I}^{F}(h)}   \ar[l]_{\tau^{\Btilde}} \\
F(\Bhat)   \ar[r]_{\Fra^{\Bhat}}       \ar@/_2pc/[rr]_{\Lambda^{\Bhat}}          & \Dom{\Bhat}{\A}   & \mathcal{I}^{F}(\Bhat)      \ar[l]^{\tau^{\Bhat}}
}\]

where $h\colon \Dom{\Btilde}{\A} \to \Dom{\Bhat}{\A}$ is the restriction of $h\colon \Btilde^{<\om} \to \Bhat^{<\om}$, which is the extension of $h\colon \Btilde \to \Bhat$.

The right-hand square commutes by definition of $\mathcal{I}^{F}(h)$. 
To show that the left-hand square commutes, take $i\in F(\Btilde)$ and $j=F(h)(i)\in F(\Bhat)$.
Let $(\abar,i)=\Fra^{\Btilde}(i)\in \Dom{\Btilde}{\A}$ and $(\bbar,j)=\Fra^{\Bhat}(j)\in \Dom{\Bhat}{\A}$.
We need to show that $h(\abar, i) \sim^{\Bhat} (\bbar,j)$.
Observe that $h(\abar,i)=(h(\abar),i)$.

With $\Phi_*^{D(\Btilde)\oplus h\oplus D(\Bhat)}(i)=j$,
we can make $\Phi_*^{D_{\Btilde}(\abar)\oplus h\upto|\abar|\oplus D_{\Bhat}(\bbar)}(i)=j$
by extending $\abar$ and $\bbar$.
Since $D_{\Btilde}(\abar)=D_{\Bhat}(h(\abar))$, we get that that $(h(\abar),i)\sim (\bbar,j)$ in $\Bhat$ as needed.
\end{proof}

\section{Proof of  Theorem \ref{thm:bi-interpretability}}\label{sec:gener}

Before proving Theorem \ref{thm:bi-interpretability}, we will prove the alternate characterization of bi-interpretations which is independent of the choice of $f^\A_\B$ and $f^\B_\A$. Throughout this section we will use the following convention. Given a map $h$ with domain $\A$, $h$ induces a map on tuples, and hence a map on $\Do_\B^\A$. We will denote this induced map by $\tilde{h}$, and the map induced on $\Do_\A^{\Do_\B^\A}$ by $\tilde{\tilde{h}}$. For example, if $h \colon \Do_\B^\A \to \A$ is a map, then $\tilde{h}$ is a map $\Do_\A^{(\Do_\B^\A)} \to \Do_{\B}^\A$.

\begin{proposition}\label{prop:char}
Let $\A$ and $\B$ be computable structures. Suppose that $\A$ is effectively interpretable in $\B$ and $\B$ is effectively interpretable in $\A$, and let $F$ and $G$ be the functors obtained from these interpretations. Then the following are equivalent.
\begin{enumerate}
	\item $\A$ and $\B$ are effectively bi-interpretable using the interpretations above.
	\item There are u.r.i.\ computable isomorphisms $g \colon \Do_\A^{(\Do_\B^\A)} \to \A$ and $h \colon \Do_\B^{(\Do_\A^\B)} \to \B$, along with isomorphisms $\alpha \colon \Do_\A^\B \to \A$ and $\beta \colon \Do_\B^\A \to \B$, such that $\alpha \circ \tilde{h} \circ \tilde{\tilde{\alpha}}^{-1} = g$ and $\beta \circ \tilde{g} \circ \tilde{\tilde{\beta}}^{-1} = h$.
	\item There are u.r.i.\ computable isomorphisms $g \colon \Do_\A^{(\Do_\B^\A)} \to \A$ and $h \colon \Do_\B^{(\Do_\A^\B)} \to \B$ such that, for all isomorphisms $\alpha \colon \Do_\A^\B \to \A$ and $\beta \colon \Do_\B^\A \to \B$, we have $\alpha \circ \tilde{h} \circ \tilde{\tilde{\alpha}}^{-1} = g$ and $\beta \circ \tilde{g} \circ \tilde{\tilde{\beta}}^{-1} = h$.
\end{enumerate}
In (2) and (3), one may always take $g$ and $h$ to be the u.r.i.\ computable maps
from the bi-interpretation in (1).
\end{proposition}
\begin{proof}
(1)$\Rightarrow$(2). Suppose that $\A$ and $\B$ are effectively bi-interpretable; then the compositions
\[
f^\A_\B \circ \tilde{f}^\B_\A\colon \Do_\B^{(\Do_\A^\B)} \to \B 
\quad\mbox{ and }\quad
f^\B_\A \circ \tilde{f}^\A_\B\colon \Do_\A^{(\Do_\B^\A)} \to \A 
\]
are u.r.i.\ computable. Take $g = f^\B_\A \circ \tilde{f}^\A_\B$ and $h = f^\A_\B \circ \tilde{f}^\B_\A$. Let $\alpha = f^\B_\A$ and $\beta = f^\A_\B$. Then
\[ \alpha \circ \tilde{h} \circ \tilde{\tilde{\alpha}}^{-1} = f^\B_\A \circ \tilde{f}^\A_\B \circ \tilde{\tilde{f}}^\B_\A (\tilde{\tilde{f}}^\B_\A)^{-1} = f^\B_\A \circ \tilde{f}^\A_\B = g \]
and similarly $\beta \circ \tilde{g} \circ \tilde{\tilde{\beta}}^{-1} = h$.

(2)$\Rightarrow$(3). Let $g \colon \Do_\A^{(\Do_\B^\A)} \to \A$ and $h \colon \Do_\B^{(\Do_\A^\B)} \to \B$ be as in (2), with isomorphisms $\alpha \colon \Do_\A^\B \to \A$ and $\beta \colon \Do_\B^\A \to \B$ such that $\alpha \circ \tilde{h} \circ \tilde{\tilde{\alpha}}^{-1} = g$ and $\beta \circ \tilde{g} \circ \tilde{\tilde{\beta}}^{-1} = h$. Let $\alpha' \colon \Do_\A^\B \to \A$ and $\beta' \colon \Do_\B^\A \to \B$ be arbitrary isomorphisms. Let $\delta \colon \Do_\A^\B \to \Do_\A^\B$ be such that $\alpha' \circ \delta = \alpha$.
Then
\[ g = \alpha \circ \tilde{h} \circ (\tilde{\tilde{\alpha}})^{-1} = \alpha' \circ \delta \circ \tilde{h} \circ (\tilde{\tilde{\delta}})^{-1} \circ (\tilde{\tilde{\alpha'}})^{-1}. \] 
We claim that $\delta \circ \tilde{h} \circ (\tilde{\tilde{\delta}})^{-1} = \tilde{h}$, and hence that $g = \alpha' \circ \tilde{h} \circ (\tilde{\tilde{\alpha'}})^{-1}$. Using $h$, we can extend $\delta$ to an automorphism $\gamma=h\circ\tilde{\delta}\circ h^{-1}$ of $\B$, and we show below that $\tilde{\gamma}=\delta$. 
Now, since $h$ is u.r.i.\ computable, $\gamma(\Gamma_h) = \Gamma_h$ where $\Gamma_h$ is the graph of $h$. 
But this means that $\gamma \circ h \circ (\tilde{\tilde{\gamma}})^{-1} = h$. Taking tildes
of both sides then shows that $\delta \circ \tilde{h} \circ (\tilde{\tilde{\delta}})^{-1} = \tilde{h}$ as required.

To see that $\tilde{\gamma}=\delta$, notice that
$$ \text{id} = \gamma^{-1}\circ h\circ\tilde{\delta}\circ h^{-1}
=h\circ(\tilde{\tilde{\gamma}})^{-1}\circ\tilde{\delta}\circ h^{-1},$$
so $\tilde{\tilde{\gamma}}=\tilde{\delta}$.  
Let
$$ \ghat = 
(\alpha)^{-1}\circ g\circ\tilde{\tilde{\alpha}}: \Do_\A^{\Do_\B^{\Do_\A^\B}} \to \Do_\A^\B.$$
Now $\ghat$ must be u.r.i.\ computable in $\B$, since $g$ is (in $\A$) and since the structure of $\Do_\A^\B$ is $\Sigma_1^c$-defined in $\B$.
This yields
$$\tilde{\gamma} = \ghat\circ\tilde{\tilde{\tilde{\gamma}}}\circ(\ghat)^{-1}
= \ghat\circ\tilde{\tilde{\delta}}\circ(\ghat)^{-1} =\delta,
$$
since $\delta$ induces (via $\alpha$) an automorphism of $\A$, which fixes the graph $\Gamma_g$.

A similar argument shows that
$h=\beta' \circ \tilde{g} \circ (\tilde{\tilde{\beta'}})^{-1}$.

(3)$\Rightarrow$(1). Let $g \colon \Do_\A^{(\Do_\B^\A)} \to \A$ and $h \colon \Do_\B^{(\Do_\A^\B)} \to \B$ be as in (3). Fix an isomorphism $f^\A_\B \colon \Do_\B^\A \to \B$. Let $f^\B_\A \colon \Do_\A^\B \to \A$ be $g \circ (\tilde{f}^\A_\B)^{-1}$, so that $f^\B_\A \circ \tilde{f}^\A_\B = g$. Then
\[ h = f^\A_\B \circ \tilde{g} \circ (\tilde{\tilde{f}}^\A_\B)^{-1} 
= f^\A_\B \circ \tilde{f}^\B_\A \circ \tilde{\tilde{f}}^\A_\B \circ (\tilde{\tilde{f}}^\A_\B)^{-1} 
= f^\A_\B \circ \tilde{f}^\B_\A.\]
Thus $f^\B_\A \circ \tilde{f}^\A_\B$ and $f^\A_\B \circ \tilde{f}^\B_\A$ are u.r.i.\ computable.
\end{proof}

Recall theorem \ref{thm:bi-interpretability} that says that $\A$ and $\B$ are effectively bi-interpretable  iff $\A$ and $\B$ are computably bi-transformable.

 \begin{proof}[Proof of  Theorem \ref{thm:bi-interpretability}]
Suppose $\A$ and $\B$ are effectively bi-interpretable. From the interpretation of $\B$ in $\A$, we get a computable functor $F = (\Phi,\Phi_*)$ from $\Iso{\A}$ to $\Iso{\B}$ which arises by exactly the process described in the proof of Proposition \ref{pr1}. Recall again from the proof of Proposition \ref{pr1} that for each $\Atilde \in \Iso{\A}$ we build $F(\Atilde)$ out of the interpretation of $\B$ within $\A$ by pulling back through a bijection $\tilde{\tau}\colon \om\to \Dom{\A}{\B}$. Then $\tilde{\tau}$ is an isomorphism $F(\Atilde) \to \Dom{\A}{\B} / \sim$ and we remarked that it was given by a computable functional in $\Atilde$. So there is a computable functional $\Omega$ with $\Omega^{\Atilde} \colon \Dom{\Atilde}{\B} \to F(\Atilde)$ (note that $\Omega$ gives the inverse of $\tau$). Similarly, there is a computable functor $G=(\Psi,\Psi_{*})$ from $\Iso{\B}$ to $\Iso{\A}$ and a computable functional $\Gamma$ with $\Gamma^{\Btilde} \colon \Dom{\Btilde}{\A} \to G(\Btilde)$. We will show that $F$ and $G$ are pseudo-inverses. We begin by showing that $G \circ F \colon \Iso{\A} \to \Iso{\A}$ is effectively isomorphic to the identity functor.

The u.r.i.\ computable map $f_{\A}^{\B} \circ \tilde{f}_{\B}^{\A} \colon \Dom{\Dom{\A}{\B}}{\A} \to \A$ gives rise to a computable functional $\Theta$ which gives isomorphisms $\Theta^{\Atilde} \colon \Atilde \to \Dom{\Dom{\Atilde}{\B}}{\A}$. 

Given $\Atilde \in \Iso{\A}$, define $\Lambda^{\Atilde}$ as follows. We have the following maps:
\[\xymatrix{\Atilde\ar@/^2pc/[rr]^{\Theta^{\Atilde}}\ar@{-->}[d]_{F}\ar@{}[r]|-*[@]{\hookleftarrow} & \Dom{\Atilde}{\B}\ar@{}[r]|-*[@]{\hookleftarrow}\ar[dl]^{\Omega^{\Atilde}} & \Dom{\Dom{\Atilde}{\B}}{\A}\ar[dl]^{\widetilde{\Omega}^{\Atilde}}\\
F(\Atilde)\ar@{-->}[d]^{G}\ar@{}[r]|-*[@]{\hookleftarrow} & \Dom{F(\Atilde)}{\A}\ar[dl]^{\Gamma^{F(\Atilde)}}\\
G(F(\Atilde))
}\]
where $\widetilde{\Omega}^{\Atilde}$ is the extension of $\Omega^{\Atilde}$ to tuples.
Let $\Lambda^{\Atilde}$ be the composition
\[ \Lambda^{\Atilde} = \Gamma^{F(\Atilde)} \circ \widetilde{\Omega}^{\Atilde} \circ \Theta^{\Atilde}. \]
We will show that $\Lambda$ is the Turing functional which witnesses that $G \circ F$ is effectively isomorphic to the identity functor. We must show that the diagram from Definition \ref{def:effeq} commutes.

Now given $j: \Atilde \to \Ahat$, we have maps as shown in the following diagram (which has
not yet been seen to commute):
\[\xymatrix@C=0.75em{\Dom{\Dom{\Atilde}{\B}}{\A}\ar@{}[r]|-*[@]{\subseteq}\ar[dr]_{\widetilde{\Omega}^{\Atilde}} & \Dom{\Atilde}{\B}\ar@{}[r]|-*[@]{\subseteq}\ar[dr]^{\Omega^{\Atilde}} & \Atilde\ar@/_2pc/[ll]_{\Theta^{\Atilde}}\ar@{-->}[d]^{F}\ar[rr]^{j}\ar@/^2pc/[dd]^(.25){\Lambda^{\Atilde}}|\hole && \Ahat\ar@/_2pc/[dd]_(.75){\Lambda^{\Ahat}}|\hole\ar@/^2pc/[rr]^{\Theta^{\Ahat}}\ar@{-->}[d]_{F}\ar@{}[r]|-*[@]{\supseteq} & \Dom{\Ahat}{\B}\ar@{}[r]|-*[@]{\supseteq}\ar[dl]_{\Omega^{\Ahat}} & \Dom{\Dom{\Ahat}{\B}}{\A}\ar[dl]^{\widetilde{\Omega}^{\Ahat}}\\
 & \Dom{F(\Atilde)}{\A}\ar@{}[r]|-*[@]{\subseteq}\ar[dr]_{\Gamma^{F(\Atilde)}} & F(\Atilde)\ar@{-->}[d]^{G}\ar[rr]_{F(j)} && F(\Ahat)\ar@{-->}[d]_{G}\ar@{}[r]|-*[@]{\supseteq} & \Dom{F(\Ahat)}{\A}\ar[dl]^{\Gamma^{F(\Ahat)}}\\
 &  & G(F(\Atilde))\ar[rr]_{G(F(j))} && G(F(\Ahat))
}\]

By definition (see Proposition \ref{pr1}) we have that
\[ G(F(j)) = \Gamma^{F(\Ahat)} \circ \widetilde{F(j)} \circ (\Gamma^{F(\Atilde)})^{-1} \]
and
\[ F(j) = \Omega^{\Ahat} \circ \tilde{j} \circ (\Omega^{\Atilde})^{-1}. \]
Hence
\[G(F(j)) \circ \Gamma^{F(\Atilde)} \circ \widetilde{\Omega}^{\Atilde} = \Gamma^{F(\Ahat)} \circ \widetilde{\Omega}^{\Ahat} \circ \tilde{\tilde{j}}. \]
Also, since $\Theta$ is u.r.i.\ computable on $\A$, for any isomorphism $j: \Atilde \to \Ahat$, we have that $\tilde{\tilde{j}} \circ \Theta^{\Atilde} = \Theta^{\Ahat} \circ j$. Hence
\[G(F(j)) \circ \Gamma^{F(\Atilde)} \circ \widetilde{\Omega}^{\Atilde} \circ \Theta^{\Atilde} = \Gamma^{F(\Ahat)} \circ \widetilde{\Omega}^{\Ahat} \circ \Theta^{\Ahat} \circ j. \]
Using the definition of $\Lambda^{\Atilde}$, we have
\[G(F(j)) \circ \Lambda^{\Atilde} = \Lambda^{\Ahat} \circ j.\]
Thus $G \circ F$ is effectively isomorphic to the identity functor via $\Lambda$.

By a similar argument, $F \circ G$ is effectively isomorphic to the identity functor. Denote the $\Lambda$ obtained for $G \circ F$ as $\Lambda_{\A}$, and that for $F \circ G$ as $\Lambda_{\B}$. Let $\Upsilon$ be the Turing functional which arises from the u.r.i.\ computable isomorphism $f_{\B}^{\A} \circ \tilde{f}_{\A}^{\B}$,
so $\Upsilon^\B\colon \B\to \Dom{\Dom{\B}{\A}}{\B}$. Then
\[ \Lambda_{\B}^{\Btilde} = \Omega^{G(\Btilde)} \circ \widetilde{\Gamma}^{\Btilde} \circ \Upsilon^{\Btilde}. \]
Then
\[ F(\Lambda_\A^{\Atilde}) = \Omega^{G(F(\Atilde))} \circ \widetilde{\Gamma}^{F(\Atilde)} \circ \widetilde{\widetilde{\Omega}}^{\Atilde} \circ \widetilde{\Theta}^{\Atilde} \circ (\Omega^{\Atilde})^{-1}. \]
Now by Proposition \ref{prop:char} with $h^{-1}=\Upsilon^{F(\Atilde)}$, 
$g^{-1}=\Theta^{\Atilde}$, and $\beta=\Omega^{\Atilde}$, we have:
\[ \Upsilon^{\F(\Atilde)} = \widetilde{\widetilde{\Omega}}^{\Atilde} \circ \widetilde{\Theta}^{\Atilde} \circ (\Omega^{\Atilde})^{-1} \]
and so
\[ F(\Lambda_\A^{\Atilde}) = \Lambda_{\B}^{F(\Atilde)}. \]
A similar argument shows that
\[ G(\Lambda_\B^{\Btilde}) = \Lambda_{\A}^{G(\Btilde)}. \]

\medskip{}

Now suppose that we have computable functors $F$ and $G$ which give a computable bi-transformation between $\A$ and $\B$. Let $\Lambda^{\Atilde} \colon \Atilde \to G(F(\Atilde)) $ witness that $G \circ F$ is effectively isomorphic to the identity. From $F$ and $G$ we get interpretations of $\A$ in $\B$ and of $\B$ in $\A$, and Turing functionals $\Omega$ and $\Gamma$ as before. For any $\Atilde \in \Iso{\A}$, we get an isomorphism
\[ \Theta^{\Atilde} = (\widetilde{\Omega}^{\Atilde})^{-1} \circ (\Gamma^{F(\Atilde)})^{-1} \circ \Lambda^{\Atilde} \colon \Atilde \to \Dom{\Dom{\Atilde}{\B}}{\A}. \]
We can view $\Theta^{\Atilde}$ as a subset of $\Atilde \times \Dom{\Dom{\Atilde}{\B}}{\A}$. 

First, let $j: \A \to \Atilde$ be any isomorphism. We show that the graph of $\Theta^{\Atilde}$ is the image, under $j$, of the graph of $\Theta^{\A}$, i.e.\ that $\Theta^{\Atilde} \circ j = \tilde{\tilde{j}} \circ \Theta^{\A}$. This is very similar to the argument above. By the properties of $\Lambda$ we have
\[ \Lambda^{\Atilde} \circ j = G(F(j)) \circ \Lambda^{\A}  = \Gamma^{F(\Atilde)} \circ \widetilde{\Omega}^{\Atilde} \circ\tilde{\tilde{j}} \circ (\widetilde{\Omega}^{\A})^{-1} \circ (\Gamma^{F(\A)})^{-1} \circ \Lambda^{\A}.\]
Then
\[ (\widetilde{\Omega}^{\Atilde})^{-1} \circ (\Gamma^{F(\Atilde)})^{-1} \circ \Lambda^{\Atilde} \circ j = \tilde{\tilde{j}} \circ (\widetilde{\Omega}^{\A})^{-1} \circ (\Gamma^{F(\A)})^{-1} \circ \Lambda^{\A} \]
which gives
$\Theta^{\Atilde} \circ j = \tilde{\tilde{j}} \circ \Theta^{\A}.$

This argument shows first that $\Theta^\A$ is fixed under automorphisms $j\colon\A\to\A$,
hence $\mathcal{L}_{\omega_1\omega}$-definable.  The same argument also shows
(with $j\colon\A\to\Atilde$ any isomorphism) that the same formula also
defines $\Theta^{\Atilde}$.  
But $\Theta$ is a Turing functional, so membership in $\Theta^{\Atilde}$ is always computable
below $\Atilde$, and so $\Theta^\A$ is u.r.i.\ computable.

\comment{
We claim that $\Theta^{\A} \subseteq \A \times \Dom{\Dom{\A}{\B}}{\A}$ is u.r.i.\ computable. This follows from two claims: first, that $\Theta^{\A}$ is $\mathcal{L}_{\omega_1 \omega}$-definable, and second that for all $\Atilde \in \Iso{\A}$, the graph of $\Theta^{\Atilde}$ is defined by the same formula. By a theorem of Ash, Knight, Manasse, and Slaman \cite{AKMS89} and Chisholm \cite{Chi90}, since $\Theta^{\Atilde}$ is always computable uniformly in $\Atilde$, it suffices to show that $\Theta^{\A}$ is uniformly u.r.i.\ computable.

We have to show that for any $\Atilde \in \Iso{\A}$, $\Theta^{\Atilde}$ is defined by the same formula as $\Theta^{\A}$. 

Now taking $\Atilde = \A$ above and $j: \A \to \A$ an automorphism of $\A$, we get
\[ \Theta^{\A} \circ j = j \circ \Theta^{\A}. \]
That is, $\Theta^{\A}$ is fixed by automorphisms of $\A$, and hence $\mathcal{L}_{\omega_1 \omega}$-definable.
}

A similar argument works to define $\Upsilon^{\Btilde} \colon \Btilde \to \Dom{\Dom{\Btilde}{\A}}{\B}$. Let $\Lambda_{\A}^{\Atilde} \colon \Atilde \to G(F(\Atilde))$ now denote the Turing functional which witnesses that $G \circ F$ is effectively isomorphic to the identity, and let $\Lambda_{\B}^{\Btilde} \colon \Btilde \to F(G(\Btilde))$ denote the Turing functional which witnesses that $F \circ G$ is effectively isomorphic to the identity. We claim that (2) of Proposition 
 \ref{prop:char} is satisfied by 
$h^{-1}=\Upsilon^{F(\Atilde)}$, 
$g^{-1}=\Theta^{\Atilde}$, 
$\alpha=\Gamma^{F(\Atilde)}$, and $\beta=\Omega^{\Atilde}$.

We have
\[ \Upsilon^{F(\Atilde)} = (\widetilde{\Gamma}^{F(\Atilde)})^{-1} \circ (\Omega^{G(F(\Atilde))})^{-1} \circ \Lambda_{\B}^{F(\Atilde)}.\]
Then
\[ \widetilde{\Omega}^{\Atilde} \circ \widetilde{\Theta}^{\Atilde} \circ (\Omega^{\Atilde})^{-1} = (\widetilde{\Gamma}^{F(\Atilde)})^{-1} \circ \widetilde{\Lambda}^{\Atilde}_\A \circ (\Omega^{\Atilde})^{-1}.\]
Now
\[ F(\Lambda^{\Atilde}_\A) = \Omega^{G(F(\Atilde))} \circ \widetilde{\Lambda}^{\Atilde}_\A \circ (\Omega^{\Atilde})^{-1} \]
and so
\[ \Omega^{\Atilde} \circ \widetilde{\Theta}^{\Atilde} \circ (\Omega^{\Atilde})^{-1} = (\widetilde{\Gamma}^{F(\Atilde)})^{-1} \circ (\Omega^{G(F(\Atilde))})^{-1} \circ F(\Lambda^{\Atilde}_\A).\]
Since $F(\Lambda^{\Atilde}_\A) = \Lambda^{F(\Atilde)}_\B$, $\Omega^{\Atilde} \circ \widetilde{\Theta}^{\Atilde} \circ (\Omega^{\Atilde})^{-1} = \Upsilon^{F(\Atilde)}$. Similarly, we get that $\Gamma^{\Btilde} \circ \widetilde{\Upsilon}^{\Btilde} \circ (\Gamma^{\Btilde})^{-1} = \Theta^{\Atilde}$. By Proposition \ref{prop:char}, we get a bi-interpretation.
\end{proof}

\bibliographystyle{alpha}
\bibliography{References}

\newcommand{\etalchar}[1]{$^{#1}$}
\def\cprime{$'$} \def\cprime{$'$}
\begin{thebibliography}{HKSS02b}

\bibitem[AKMS89]{AKMS89}
Chris Ash, Julia Knight, Mark Manasse, and Theodore Slaman.
\newblock Generic copies of countable structures.
\newblock {\em Ann. Pure Appl. Logic}, 42(3):195--205, 1989.

\bibitem[Chi90]{Chi90}
John Chisholm.
\newblock Effective model theory vs.\ recursive model theory.
\newblock {\em J. Symbolic Logic}, 55(3):1168--1191, 1990.

\bibitem[Ers96]{Ershov96}
Yuri~L. Ershov.
\newblock {\em Definability and computability}.
\newblock Siberian School of Algebra and Logic. Consultants Bureau, New York,
  1996.

\bibitem[HKSS02a]{HKSS}
Denis~R. Hirschfeldt, Bakhadyr Khoussainov, Richard~A. Shore, and Arkadii~M.
  Slinko.
\newblock Degree spectra and computable dimensions in algebraic structures.
\newblock {\em Ann. Pure Appl. Logic}, 115(1-3):71--113, 2002.

\bibitem[HKSS02b]{HKSS02}
Denis~R. Hirschfeldt, Bakhadyr Khoussainov, Richard~A. Shore, and Arkadii~M.
  Slinko.
\newblock Degree spectra and computable dimensions in algebraic structures.
\newblock {\em Ann. Pure Appl. Logic}, 115(1-3):71--113, 2002.

\bibitem[Hod93]{Hodges93}
Wilfrid Hodges.
\newblock {\em Model theory}, volume~42 of {\em Encyclopedia of Mathematics and
  its Applications}.
\newblock Cambridge University Press, Cambridge, 1993.

\bibitem[Kal09]{Kal09}
I.~Sh. Kalimullin.
\newblock Relations between algebraic reducibilities of algebraic systems.
\newblock {\em Izv. Vyssh. Uchebn. Zaved. Mat.}, 53(6):71--72, 2009.

\bibitem[Mar02]{MarkerBook}
David Marker.
\newblock {\em Model theory}, volume 217 of {\em Graduate Texts in
  Mathematics}.
\newblock Springer-Verlag, New York, 2002.
\newblock An introduction.

\bibitem[MK08]{MK08}
Andrei~S. Morozov and Margarita~V. Korovina.
\newblock On {$\Sigma$}-definability without equality over the real numbers.
\newblock {\em MLQ Math. Log. Q.}, 54(5):535--544, 2008.

\bibitem[Mon]{MonICM}
Antonio Montalb\'an.
\newblock Computability theoretic classifications for classes of structures.
\newblock To appear in the Proccedings of the ICM 2014.

\bibitem[Mon12]{MonRice}
Antonio Montalb{\'a}n.
\newblock Rice sequences of relations.
\newblock {\em Philos. Trans. R. Soc. Lond. Ser. A Math. Phys. Eng. Sci.},
  370(1971):3464--3487, 2012.

\bibitem[Mon13]{MonFixed}
Antonio Montalb\'an.
\newblock A fixed point for the jump operator on structures.
\newblock {\em Journal of Symbolic Logic}, 78(2):425--438, 2013.

\bibitem[MPP{\etalchar{+}}]{MPPSS}
R.~Miller, J.~Park, B.~Poonen, H.~Schoutens, and A.~Shlapentokh.
\newblock A computable functor from graphs to fields.
\newblock To appear.

\bibitem[Puz09]{Puz09}
V.~G. Puzarenko.
\newblock On a certain reducibility on admissible sets.
\newblock {\em Sibirsk. Mat. Zh.}, 50(2):415--429, 2009.

\bibitem[Stu13]{StuEMU}
Alexey Stukachev.
\newblock Effective model theory: an approach via {$\Sigma$}-definability.
\newblock In {\em Effective mathematics of the uncountable}, volume~41 of {\em
  Lect. Notes Log.}, pages 164--197. Assoc. Symbol. Logic, La Jolla, CA, 2013.

\end{thebibliography}

\end{document}